\documentclass[12pt,a4paper]{amsart}

\usepackage{mathtools,amssymb,amsmath,amsopn,amsthm,graphicx,MnSymbol,centernot,stmaryrd}
\usepackage{tikz}
\usepackage{enumitem}
\usepackage{etoolbox}
\usepackage{verbatim}
\usepackage{hyperref}
\hypersetup{
    colorlinks=true,
    linkcolor=blue,
    filecolor=magenta,      
    urlcolor=cyan,}
\patchcmd{\subsection}{-.5em}{.5em}{}{}
\usetikzlibrary{matrix}
\begin{document}

\newtheorem{definition}{Definition}[section]
\newtheorem{definitions}[definition]{Definitions}
\newtheorem{deflem}[definition]{Definition and Lemma}
\newtheorem{lemma}[definition]{Lemma}
\newtheorem{proposition}[definition]{Proposition}
\newtheorem{theorem}[definition]{Theorem}
\newtheorem{corollary}[definition]{Corollary}
\newtheorem{cors}[definition]{Corollaries}
\theoremstyle{remark}
\newtheorem{remark}[definition]{Remark}
\theoremstyle{remark}
\newtheorem{remarks}[definition]{Remarks}
\theoremstyle{remark}
\newtheorem{notation}[definition]{Notation}
\theoremstyle{remark}
\newtheorem{example}[definition]{Example}
\theoremstyle{remark}
\newtheorem{examples}[definition]{Examples}
\theoremstyle{remark}
\newtheorem{dgram}[definition]{Diagram}
\theoremstyle{remark}
\newtheorem{fact}[definition]{Fact}
\theoremstyle{remark}
\newtheorem{illust}[definition]{Illustration}
\theoremstyle{remark}
\newtheorem{rmk}[definition]{Remark}
\theoremstyle{definition}
\newtheorem{que}[definition]{Question}
\theoremstyle{definition}
\newtheorem{conj}[definition]{Conjecture}
\newtheorem{por}[definition]{Porism}

\renewenvironment{proof}{\noindent {\bf{Proof.}}}{\hspace*{3mm}{$\Box$}{\vspace{9pt}}}

\author[G. Jindal]{Gorav Jindal}
\address{Department of Computer Science, Aalto University, Finland. Supported by European Research Council (ERC) under the European Union’s Horizon 2020 research and innovation program (grant agreement No 759557) and by Academy of Finland, under grant number 310415}
\email{\href{gorav.jindal@gmail.com}{gorav.jindal@gmail.com}}

\author[A. Pandey]{Anurag Pandey}
\address{Max-Planck Institute for Computer Science, Saarland Informatics Campus, Saarbr\"ucken, Germany}
\email{\href{apandey@mpi-inf.mpg.de}{apandey@mpi-inf.mpg.de}}

\author[H. Shukla]{Himanshu Shukla}
\address{Max-Planck Institute for Computer Science, Saarland Informatics Campus, Saarbr\"ucken, Germany}
\email{\href{hshukla@mpi-inf.mpg.de}{hshukla@mpi-inf.mpg.de}}

\author[C. Zisopoulos]{Charilaos Zisopoulos}
\address{Department of Computer Science, Saarland Informatics Campus, Saarbr\"ucken, Germany}
\email{\href{s9chziso@stud.uni-saarland.de}{s9chziso@stud.uni-saarland.de}}


\title{{How many zeros of a random \textit{sparse} polynomial are real?}}
\keywords{Algebraic complexity theory, sparse polynomials, real-tau conjecture, random polynomials, real zeros}
\subjclass[2010]{68Q15, 34F05, 11C08, 30C15, 26D99}

\maketitle

\begin{abstract}
We investigate the number of real zeros of a univariate $k$-sparse polynomial $f$ over the reals, when the coefficients of $f$ come from independent standard normal distributions. Recently B\"urgisser, Erg\"ur and Tonelli-Cueto showed that the expected number of real zeros of $f$ in such cases is bounded by $O(\sqrt{k} \log k)$. In this work, we improve the bound to $O(\sqrt{k})$ and also show that this bound is tight by constructing a family of sparse support whose expected number of real zeros is lower bounded by $\Omega(\sqrt{k})$. Our main technique is an alternative formulation of the Kac integral by Edelman-Kostlan which allows us to bound the expected number of zeros of $f$ in terms of the expected number of zeros of polynomials of lower sparsity.
Using our technique, we also recover the $O(\log n)$ bound on the expected number of real zeros of a dense polynomial of degree $n$ with coefficients coming from independent standard normal distributions.
\end{abstract}

\newcommand\A{\mathcal{A}}
\newcommand\C{\mathcal{C}}
\newcommand\D{\mathcal{D}}
\newcommand\LL{\mathcal{L}}
\newcommand\M{\mathcal{M}}
\newcommand\RR{\mathcal{R}}
\newcommand\T{\mathcal{T}}
\newcommand{\card}[1]{|#1|}    
\newcommand{\dpri}[1]{#1^{\prime\prime}}
\newcommand{\pri}[1]{#1^\prime}
\newcommand{\N}{\mathbb{N}} 
\newcommand{\CC}{\mathbb{C}}
\newcommand{\R}{\mathbb{R}}
\newcommand{\Z}{\mathbb{Z}}
\newcommand{\Dphi}{\mathcal{D}^{\emptyset}_n}
\newcommand{\Dabar}{\mathcal{D}^{\overline{a}}_n}
\newcommand{\Dbbar}{\mathcal{D}^{\overline{b}}_n}
\newcommand{\Dnq}{\mathcal{D}^{Q}_n}
\newcommand{\norm}[1]{\left\lVert#1\right\rVert}
\newcommand{\Hb}{\mathbb{H}}
\newcommand{\Q}{\mathcal{Q}}
\newcommand{\la}{\langle}
\newcommand{\ra}{\rangle}
\newcommand{\Def}{\mathrm{Def}}
\newcommand{\Deft}{\widetilde{\mathrm{Def}}}
\newcommand{\Atphi}{At_n^{\emptyset}}
\newcommand{\At}{At_n^{\overline{a}}}
\newcommand{\Col}{\mathrm{Color}(n,\overline{a})}
\newcommand{\Split}{\mathrm{Split}^{\overline{a}}_n}
\newcommand{\LC}{\chi_{T}}
\newcommand{\GC}{\chi_n^{\overline{a}}}
\renewcommand{\qedsymbol}{$\blacksquare$}
\newcommand{\Rplus}{[0, \infty)}
\newcommand{\Slz}{SL\_2(\Z)}
\newcommand{\PP}{\mathcal{P}}
\newcommand{\res}{\upharpoonright}
\newcommand{\MLR}{Martin L\"of random}
\newcommand{\lep}{\stackrel{\scriptstyle +}{\le}}
\newcommand{\mmid}{\quad\mid\quad}
\newcommand*\quot[2]{{^{\textstyle #1}\big/_{\textstyle #2}}}

\section{Introduction}
\label{sec:intro}

Understanding the number of real zeros of a given real univariate polynomial has always been of interest, both from a theoretical as well as an application point of view in science, engineering and mathematics. The arithmetic of sparse polynomials has been of special interest in computer science and the algorithms for efficiently finding roots of sparse polynomials have been widely studied (see \cite{cuckerKS99, lenstra99, rouillier2004,  eigenwillig2008real, Sagraloffissac14, SagraloffMjsc16}).

\subsection{Zeros of random sparse univariate polynomials}

In order to gain a better understanding of the behavior of the number of real zeros for sparse polynomials and its generalizations, we study the case of a single univariate sparse random polynomial. For simplicity, in this article, we only consider the case when the coefficients are identically distributed independent standard normal random variables. 

With respect to this consideration, the dense case\footnote{i.e. there is no restriction on the sparsity, thus we have a polynomial $f$ of degree $n$ with all its $n+1$ coefficients as standard normal random variables} has been extensively studied and is well understood. This problem was first considered in a series of works by  Littlewood and Offord \cite{liloff,liloff43} who proved a $O(\log^2(n))$ upper bound on the expected number of real zeros in the dense case when the coefficients are from bernoulli($\{-1,1\}$), standard gaussian and uniform distributions. In 1943, Kac \cite{kac} proved that the expected number of real zeros for degree $n$ polynomial with coefficients drawn form standard normal distribution is: $$\left( \frac{2}{\pi}+o(1) \right) \log(n).$$ Later, in a series of works, Offord, Erd\H{o}s, Stevens, Ibragimov and Maslova \cite{erdos56,stevens,Ibr71} extended the results to other more general distributions including but not restricted to bernoulli (\{-1,1\}) and uniform distributions. Interested readers can look at the article by Erd\'elyi \cite{erdelyi} for more recent results in this direction. 
In 1995, Edelman and Kostlan \cite{edelmankostlan1995} gave an alternative, simpler derivation for the Kac bound using geometric methods, in addition to providing essential insights to the integral and numerous generalizations in a variety of cases. For this article, the works in \cite{kac,edelmankostlan1995} are most relevant. It seems very surprising that there are so few real zeros in the random case.

In the sparse case, there is a line of work considering the case of the multivariate system of random equations (for instance see \cite{fewnomials, rojas, rojasmala}). However their focus is  different and we are not aware of any useful adaptations to the univariate case. In fact, we do not know of any such progress until the recent work of B\"urgisser, Erg\"ur and Tonelli-Cueto \cite{burgisserergurcueto2018} which showed that for a random $k$-sparse univariate polynomial, the expected number of real roots in the standard normal case, is bounded by $\frac{4}{\pi} \sqrt{k} \log k$,\footnote{unless stated otherwise, the base of the logarithm in this article is always $e$} thus showing that in this setting, the number of real zeros is much less than the Descartes' bound. 

\subsection{Zeros of sparse polynomials}
A lot of the polynomials that we encounter in practice are \textit{sparse}, i.e. the number of monomials in them is considerably smaller than the degree of the polynomial. This motivates one to study the question for the sparse polynomials. Descartes' famous rule of signs from the 17th century \cite{descartes1886geometrie} already sheds some light by bounding the number of non-zero real zeros of a $k$-sparse $f  \in \R[x]$ \footnote{throughout this article, polynomials considered are over reals and have degree $n$ with $n>>k$. } by $2k-2$. There are polynomials which achieve this bound too.
Having some understanding on the number of real roots of $k$-sparse polynomials, it makes sense to ask the same question for generalizations. However, if we consider the first non-trivial generalization, i.e. if we consider the polynomial $fg+1$, where $f$ and $g$ are both $k$-sparse, our understanding seems very limited. To the best of our knowledge, no bound better than the one given by Descartes' rule of sign is known in this case, in particular, no sub-quadratic bound is known. We also do not know of any example where the number of real roots of $fg+1$ is super-linear in $k$.

\subsection{Connections to algebraic complexity theory: Real Tau Conjecture}

A strong motivation for computer scientists to consider generalizations like the above was provided in 2011 by Koiran \cite{koiran11}, when he considered the number of real zeros of the sum of products of sparse polynomials. He formulated \textit{the real $\tau$-conjecture} claiming that if a polynomial is given as 
\[ f= \sum_{i =1}^{m} \prod_{j=1}^{t} f_{ij} \] where all $f_{ij}$'s are $k$-sparse, then the number of real zeros of $f$ is bounded by a polynomial in $mkt$. Thus the conjecture claims that a univariate polynomial computed by a depth-4 arithmetic circuit (see \cite{sy10, saptharishi2015survey} for background on arithmetic circuits) with the fan-in of gates at the top three layers being bounded by $m, t$ and $k$ respectively will have $O((mkt)^c)$ real zeros for some positive constant $c$. Notice that applying Descartes' bound only gives an exponential bound on the number of real zeros of $f$, since a-priori the sparsity bound that we can achieve for $f$ is only $O(mk^{t})$. 

What is of particular interest is the underlying connection of this conjecture to the central question of algebraic complexity theory. Koiran showed that proving the conjecture implies a superpolynomial lower bound on the arithmetic circuit complexity of the permanent, hence establishing the importance of the question of understanding real roots of sparse polynomials from the perspective of theory of computation as well. In fact this connection is what inspired the authors to investigate the problems considered in this article.

The real $\tau$-conjecture itself was inspired by the Shub and Smale's $\tau$-conjecture \cite{tauconj} which asserts that the number of integer zeros of a polynomial with arithmetic circuit complexity bounded by $s$ will be bounded by a polynomial in $s$. 
This conjecture also implies a super-polynomial lower bound on the arithmetic circuit size of the permanent \cite{burgisserdefining} and also implies $\mathrm{P}_{\CC} \neq \mathrm{NP}_{\CC}$ in the Blum-Shub-Smale model of computation (see \cite{tauconj, bss}). 
Koiran's motivation was to connect the complexity theoretic lower bounds to the number of real zeros instead of the number of integer zeros, because the latter takes one to the realm of number theory where problems become notoriously hard very easily.

While the real $\tau$-conjecture remains open (see \cite{hrubes13, newton, wronskian} for some works towards it), Briquel and B\"urgisser \cite{briquel18} showed that the conjecture is true in the average case, i.e. they show that when the coefficients involved in the description of $f$ are independent Gaussian random variables, then the expected number of real zeros of $f$ is bounded by $O(mkt)$.

\subsection{Our contributions}
Before we state our results we set up some notations.
Consider a set $S=\{e_1,\ldots,e_k\} \subseteq \N$ of natural numbers. For
such a set $S$, one asks how many roots (in expectation) of the random
polynomial $f_{S}=\sum_{i=1}^{k}a_{i}x^{e_{i}}$ (here $a_{i}$'s
are independent standard normals) are real. For an open interval $I\subseteq\R$,
we use $z_{S}^{I}$ to denote the expected number of roots of $f_{S}$
in $I$.  To avoid some degeneracy issues, we always assume
$0\not\in I$, this assumption allows us to assume that the smallest element
of $S$ is zero. In this paper, we are only concerned with the case
when $I=(0,1)$. See Remark \ref{rem:rem1} on why this is sufficient. When $I=(0,1)$, we simply use $z_{S}$ to denote $z_{S}^{I}$.

Our main contribution is the improvement on the bound on the expected number of real zeros of a random $k$-sparse polynomial $f$ and proving that this is the best one can do. 

\begin{theorem}\label{theorem:main} 
Let $S\subseteq\N$ be any set as
above with $\card S=k$, then we have $
z_{S}\leq\frac{2}{\pi}\sqrt{k-1}.$
\end{theorem}

\begin{remark} \label{rem:rem1}
Since our bound in Theorem \ref{theorem:main} only depends on the size of $S$, and not on the structure of $S$, we get that $z_{S}^{\R} = 4z_{S}^{(0,1)}$. For $S=\{e_1, \ldots, e_k\}$, $z_{S}^{(1,\infty)}$ is equal to $z_{S'}^{(0,1)}$ for $S'=\{n-e_1, \ldots, n-e_k\}$  by replacing $x$ by $\frac{1}{x}$ and multiplying by $x^n$, where $n$ is the degree of $f_{S}$. Also $z_{S}^{(-\infty,0)} = z_{S}^{(0, \infty)}$ by replacing $x$ by $-x$.
\end{remark}

\begin{theorem}\label{theorem:lowerbound}
There exists a sequence of sets $S_k\subset\mathbb{N}$ for $k\geq1$ with $|S_k|=k+2$ and a constant $c>0$ such that $z_{S_k}\geq c\cdot\sqrt{k}$, for large enough $k$. 
\end{theorem}

Theorem \ref{theorem:lowerbound} shows that the bound obtained in Theorem \ref{theorem:main} is tight and cannot be reduced further for an arbitrary $S\subset \mathbb{N}$.

Using our techniques, we confirm the intuition from the dense case that in expectation, all the roots are concentrated around $1$ i.e. for any small constant $\epsilon >0$, the expected number of roots in $(0,1-\epsilon)$ is bounded by a constant independent of $n$ and $k$.

\begin{theorem}
\label{theorem:rootsfarmfromone} For a fixed $\epsilon>0$ and any $S\subseteq\N$ as
above, we have $$z_{S}^{(0,1-\epsilon)}\leq\frac{1}{2\pi}\left(\log\left(\frac{2}{\epsilon}\right)+\frac{4}{\sqrt{\epsilon}}-4\right).$$
\end{theorem}

\subsection{Proof ideas}
Our main technical contribution is an alternative formulation of the Kac integral by Edelman-Kostlan, that we call the \textit{Edelman-Kostlan integral} (discussed in Section \ref{prelims}) presented in detail in Section \ref{sec:altform}. 

The formulation allows us to bound $z_{S_1 \sqcup S_2}$ in terms of the bounds on $z_{S_1}$ and $z_{S_2}$ (presented in subsection \ref{sec:setunionroots}). Thus we can build our $k$-sparse polynomial monomial-by-monomial. We show that every time we add a monomial, we do not increase the expected number of roots by a lot. A careful application of this idea yields the desired $O(\sqrt{k})$ bound (presented in Section \ref{sqrtkbound}). 

We also obtain a bound on $z_{S_1 + S_2}$ in terms of $z_{S_1}$ and $z_{S_2}$, where $S_1 + S_2$ is the set obtained as a result of the addition of elements of $S_1$ and $S_2$ (presented in subsection \ref{subsec:setadditionsroots}). Combining the bounds on $z_{S_1 + S_2}$ and $z_{S_1 \sqcup S_2}$ allows us to recover the $O(\log n)$ bound for the dense case i.e. $S = \{0,1, \ldots, n \}$, where we build up our set $S$ as a combination of unions and additions of sets (presented in Section \ref{lognbound}).

Further, the proof that all the roots are concentrated around $1$ follows from the analysis of an approximation of the Edelman-Kostlan integral. This approximation which is inspired by the one used in \cite{burgisserergurcueto2018} makes the analysis of the integral simpler.

Finally in Section \ref{lb}, we show that we cannot obtain a better bound for an arbitrary $S\subset\mathbb{N}$. We show this by applying the idea of monomial-wise construction of a polynomial (presented in Section \ref{sec:setunionroots}) on a carefully chosen monomial sequence, thus  proving Theorem \ref{theorem:lowerbound}.  
\subsection{Previous work: known bounds on \texorpdfstring{$z_{S}^I$}{the expected number of real zeroes}}
In this subsection, we present the state of the art prior to this work for $z_{S}^{I}$.

For $S=\{0,1,2,\dots,n\}$ and $I = \R$, $z_S^{I}$ is known to be bounded by $O(\log n)$.
\begin{theorem}[\cite{edelmankostlan1995, kac}]
\label{theorem:densekostlanbound}If $S=\{0,1,2,\dots,n\}$ then 
\[
z_{S}^{\R}=\frac{2}{\pi}\log(n)+C_{1}+\frac{2}{n\pi}+O\left(\frac{1}{n^{2}}\right).
\]

Here $C_{1}\approx0.6257358072\dots$.
\end{theorem}
Determining the value of $z_{S}^{I}$ for arbitrary sets $S$ remains
an open problem. Towards this the best bound known was the following result by  \cite{burgisserergurcueto2018}.
\begin{theorem}[{\cite[Theorem 1.3]{burgisserergurcueto2018}}]
\label{theorem:sparseburgisserbound}Let $S\subseteq\N$ be any set as
above with $\card S=k$ then we have
\[
z_{S}\leq\frac{1}{\pi}\sqrt{k}\log(k).
\]
\end{theorem}
For the sake of completeness, we present a proof for the above theorem in the appendix.

\section{Preliminaries}\label{prelims}

Since our method builds upon the Edelman-Kostlan method \cite{edelmankostlan1995} by a novel approach on analyzing their integral, it is essential to look at their method. In order to compute $z_{S}$ for $S = \{e_1, \ldots, e_k \}$, define a generalization of the moment curve $v_{S}$ as $v_{S}(t):=(t^{e_{1}},t^{e_{2}},\dots,t^{e_{k}})$. This allows the following expression
for $z_{S}^{I}$:
 
\begin{theorem}[{\cite[Theorem 3.1]{edelmankostlan1995}}]
\label{theorem:realrootsboundkostlan}For all sets $S\subseteq\N$, we have the following equality for $z_{S}^{I}$ .

\begin{equation}\label{zsreqinte}
z_{S}^{I}=\frac{1}{\pi}\intop_{I}\frac{\sqrt{(\norm{v_{S}(t)}_{2}\cdot\norm{\pri{v_{S}}(t)}_{2})^{2}-(v_{S}(t)\cdot\pri{v_{S}}(t))^{2}}}{(\norm{v_{S}(t)}_{2})^{2}}\mathrm{d}t.
\end{equation}
\end{theorem}
We refer to the above integral as the \textit{Edelman-Kostlan integral}.

The strength of this method is that the above integral is parameterized by the support $S$ and the interval $I$, thus allowing one to estimate the expected number of real zeros for any such arbitrary support and interval. In their paper, they compute the integral for $S = \{0,1,\ldots,k\}$ and $I=(0,1)$ and for these values showed that $z_{S}^{I}$ is bounded by $O(\log k)$. However, for arbitrary $S$
of cardinality $k$, the integral becomes quite complicated to analyze.

In \cite{burgisserergurcueto2018}, they get around this difficulty by upper bounding the integral. This is achieved by ignoring the negative term of the numerator and through some elementary norm inequalities leads to the $O(\sqrt{k} \log k)$ bound. 

We now state a basic technical proposition which will be useful in the proof of the main theorem.

\begin{proposition}
\label{lem:abcdidentity}The following identity is true for all $a,b,c,d$:

\[
\left(\frac{a+c}{b+d}\right)^{2}=\left(\frac{b}{b+d}\right)\left(\frac{a}{b}\right)^{2}+\left(\frac{d}{b+d}\right)\left(\frac{c}{d}\right)^{2}-\frac{1}{bd}\left(\frac{bc-ad}{b+d}\right)^{2}.
\]
\end{proposition}

\begin{proof}
The RHS of above equation can we written as:

\begin{align*}
\left(\frac{b}{b+d}\right)\left(\frac{a}{b}\right)^{2}+\left(\frac{d}{b+d}\right)\left(\frac{c}{d}\right)^{2}-\frac{1}{bd}\left(\frac{bc-ad}{b+d}\right)^{2} & =\frac{a^{2}d(b+d)+c^{2}b(b+d)-(bc-ad)^{2}}{(b+d)^{2}bd}\\
 & =\frac{bd(a^{2}+c^{2})+a^{2}d^{2}+c^{2}b^{2}-(bc-ad)^{2}}{(b+d)^{2}bd}\\
 & =\frac{bd(a^{2}+c^{2}+2ac)}{(b+d)^{2}bd}\\
 & =\left(\frac{a+c}{b+d}\right)^{2}.
\end{align*}
\hfill\end{proof}

\section{Alternative formulation of the Edelman-Kostlan Integral \label{sec:altform}}

In this section, we present an alternative formulation of the Edelman-Kostlan integral on which our proofs build upon. 
\begin{definition}
For a set $S=\{e_{1},e_{2},\dots,e_{k}\}\subseteq\N$, we define: $$g_{S}(t):=(\norm{v_{S}(t)}_{2})^{2}=\sum_{i=1}^{k}t^{2e_{i}}.$$
\end{definition}
In the following lemma, we show that we can express $z_{S}^{I}$ entirely in terms of $g_{S}(t)$
and its derivatives. Hence we define:
\begin{definition} \label{def:i(g)alt}
For a function $g:\R\rightarrow\R$, we define the function $\mathcal{I}(g):\R\rightarrow\R$,
\[
\mathcal{I}(g):=\pri{\left(\frac{\pri g(t)}{g(t)}\right)}+\frac{\pri g(t)}{tg(t)}=\dpri{\left(\log(g(t))\right)}+\frac{\pri{\left(\log(g(t))\right)}}{t}.
\]
\end{definition}

We are now ready give our alternative formulation. 

\begin{lemma}\label{z_sequaltorootigs}
For all sets $S\subseteq\N$ , we have the following equality for
$z_{S}^{I}$ .
\[
z_{S}^{I}=\frac{1}{2\pi}\intop_{I}\sqrt{\mathcal{I}(g_{S}(t))}\mathrm{d}t.
\]
\end{lemma}
\begin{proof}
We can rewrite equation \eqref{zsreqinte} as:
\[
z_{S}^{I}=\frac{1}{\pi}\intop_{I}\frac{\sqrt{(g_{S}(t)\cdot(\norm{\pri{v_{S}}(t)}_{2})^{2}-(v_{S}(t)\cdot\pri{v_{S}}(t))^{2}}}{g_{S}(t)}\mathrm{d}t.
\]

Now verify the following equality for $v_{S}(t)\cdot\pri{(v_{S}}(t))$.
\[
v_{S}(t)\cdot\pri{v_{S}}(t)=\sum_{i=1}^{k}e_{i}t^{2e_{i}-1}=\frac{\pri{g_{S}}(t)}{2}
\]

We also have the following equality for $(\norm{\pri{v_{S}}(t)}_{2})^{2}$.
\begin{align*}
(\norm{\pri{v_{S}}(t)}_{2})^{2} & =\sum_{i=1}^{k}e_{i}^{2}t^{2e_{i}-2}=\frac{1}{4}\cdot\left(\sum_{i=1}^{k}4e_{i}^{2}t^{2e_{i}-2}\right)\\
 & =\frac{1}{4}\cdot\left(\sum_{i=1}^{k}((2e_{i}(2e_{i}-1))+2e_{i})\cdot t^{2e_{i}-2}\right)\\
 & =\frac{1}{4}\cdot\left(\sum_{i=1}^{k}(2e_{i}(2e_{i}-1)\cdot t^{2e_{i}-2}\right)+\frac{1}{4}\cdot\left(\sum_{i=1}^{k}2e_{i}\cdot t^{2e_{i}-2}\right)\\
 & =\dpri{\frac{1}{4}g_{S}}(t)+\frac{1}{4t}\pri{g_{S}}(t).
\end{align*}

Therefore we can rewrite $z_{S}^{I}$ as:

\begin{align*}
z_{S}^{I} & =\frac{1}{\pi}\intop_{I}\sqrt{\frac{1}{4}\left(\frac{g_{S}(t)\cdot(\dpri{g_{S}}(t)+\frac{1}{t}\pri{g_{S}}(t))-(\pri{g_{S}}(t))^{2}}{(g_{S}(t))^{2}}\right)}\mathrm{d}t\\
 & =\frac{1}{2\pi}\intop_{I}\sqrt{\frac{\dpri{g_{S}}(t)}{g_{S}(t)}-\left(\frac{\pri{g_{S}}(t)}{g_{S}(t)}\right)^{2}+\frac{\pri{g_{S}}(t)}{tg_{S}(t)}}\mathrm{d}t\\
 & =\frac{1}{2\pi}\intop_{I}\sqrt{\pri{\left(\frac{\pri{g_{S}}(t)}{g_{S}(t)}\right)}+\frac{\pri{g_{S}}(t)}{tg_{S}(t)}}\mathrm{d}t\\
  & =\frac{1}{2\pi}\intop_{I}\sqrt{\dpri{\left(\log(g_{S}(t))\right)}+\frac{\pri{\left(\log(g_{S}(t))\right)}}{t}}\mathrm{d}t.
\end{align*}

\hfill\end{proof}

As can be seen from the above, whenever the Edelman-Kostlan integral is well defined, the conditions on $g$ which make $\mathcal{I}(g)$ well defined and non-negativity conditions are also satisfied. This is true for all cases we consider, i.e. for every $S$, $g_{S}(t)$ satisfies all the needed conditions.

We now give a useful characterization of $\mathcal{I}(g_S(t))$ that will be used to show that in a very small interval $I$ near 1, $z_S^I$ is very small (see Lemma \ref{lem:intI} and Remark \ref{rem:arbitsmall}).  
\begin{proposition}\label{prop:edelmanreformulation}
For $S=\{e_{1},e_{2},\dots,e_{k}\}\subseteq\N$, $\mathcal{I}(g_{S}(t))$ satisfies the following equality:
$$
\mathcal{I}(g_{S}(t)) =\frac{4}{(g_{S}(t))^{2}}\cdot\left(\sum_{c}\sum_{\substack{i<j\\
e_{i}+e_{j}-1=c
}
}((e_{i}-e_{j})t^{c})^{2}\right).
$$
\end{proposition}

\begin{proof}
We have:
\begin{align*}
(g_{S}(t)\cdot(\norm{\pri{v_{S}}(t)}_{2})^{2}-(v_{S}(t)\cdot\pri{v_{S}}(t))^{2} & =\left(\sum_{i=1}^{k}t^{2e_{i}}\right)\cdot\left(\sum_{j=1}^{k}e_{j}^{2}t^{2e_{j}-2}\right)-\left(\sum_{i=1}^{k}e_{i}t^{2e_{i}-1}\right)^{2}\\
 & =\sum_{c}\sum_{\substack{i,j\\
e_{i}+e_{j}-1=c
}
}e_{j}^{2}t^{2c}-\sum_{c}\sum_{\substack{i,j\\
e_{i}+e_{j}-1=c
}
}e_{i}\cdot e_{j}t^{2c}
\end{align*}
\begin{align*}
 & = \left(\sum_{c}\sum_{\substack{i\\
e_{i}+e_{i}-1=c
}
}e_{i}^{2}t^{2c}+\sum_{c}\sum_{\substack{i\neq j\\
e_{i}+e_{j}-1=c
}
}e_{j}^{2}t^{2c}\right)-\left(\sum_{c}\sum_{\substack{i\\
e_{i}+e_{i}-1=c
}
}e_{i}^{2}t^{2c}+\sum_{c}\sum_{\substack{i\neq j\\
e_{i}+e_{j}-1=c
}
}e_{i}e_{j}t^{2c}\right)\\
 & =\sum_{c}\sum_{\substack{i<j\\
e_{i}+e_{j}-1=c
}
}(e_{i}^{2}+e_{j}^{2}-2e_{i}e_{j})t^{2c}=\sum_{c}\sum_{\substack{i<j\\
e_{i}+e_{j}-1=c
}
}((e_{i}-e_{j})t^{c})^{2}.
\end{align*}
\hfill\end{proof}
\begin{remark}
Using the above proposition and Lemma \ref{z_sequaltorootigs} we have: 
\begin{equation}
z_S^I=\frac{1}{\pi}\intop_{I}\frac{\sqrt{\sum_{c}\sum_{\substack{i<j\\
e_{i}+e_{j}-1=c
}
}((e_{i}-e_{j})t^{c})^{2}.
}}{g_S(t)}\mathrm{d}t.
\end{equation}
\end{remark}

The formulation in Definition \ref{def:i(g)alt} allows us to prove the following proposition:

\begin{proposition}
\label{lem:Ig1g2}For two non-negative functions $g_{1},g_{2}:\R\rightarrow\R$,
we have that \textup{$\sqrt{\mathcal{I}(g_{1}\cdot g_{2})}\leq\sqrt{\mathcal{I}(g_{1})}+\sqrt{\mathcal{I}(g_{2})}$.}
\end{proposition}

\begin{proof}
Consider:

\begin{align*}
\mathcal{I}(g_{1}\cdot g_{2}) & =\dpri{\left(\log(g_{1}(t)\cdot g_{2}(t))\right)}+\frac{\pri{\left(\log(g_{1}(t)\cdot g_{2}(t))\right)}}{t}\\
 & =\dpri{\left(\log(g_{1}(t))\right)}+\frac{\pri{\left(\log(g_{1}(t))\right)}}{t}+\dpri{\left(\log(g_{2}(t))\right)}+\frac{\pri{\left(\log(g_{2}(t))\right)}}{t}\tag{By linearity of differentiation and the fact that \ensuremath{\log(g_{1}\cdot g_{2})=\log(g_{1})+\log(g_{2})}}\\
 & =\mathcal{I}(g_{1})+\mathcal{I}(g_{2}).
\end{align*}

Now the claim follows by using the fact that $\sqrt{x+y}\le\sqrt{x}+\sqrt{y}$
for non-negative $x,y$.
\hfill\end{proof}

This allows us to give a bound on the integral when $S=S_{1} * S_{2}$, where $*$ corresponds to the operation of either union or addition of sets. This bound depends on the integrals associated to the corresponding sets $S_{1}$ and $S_{2}$.

\subsection{Addition of sets\label{subsec:setadditionsroots}}

\begin{definition}
\label{def:addsets}For sets $A,B\subseteq\N$, we define the sum
of $A,B$ as: $A+B:=\{a+b:a\in A,b\in B\}$. We say two sets $A,B\subseteq\N$ are \emph{collision-free}
if $\card{A+B}=\card A\cdot\card B=\card{A\times B}$, i.e. when all
the ``$a+b:a\in A,b\in B$'' are distinct.
\end{definition}

Now we show how to apply this definition in the context of the above formulation of $z_{S}^{I}$
and $\mathcal{I}(g)$.
\begin{lemma}
\label{lem:sumofsets}If $S_{1},S_{2}\subseteq\N$ are two collision-free
sets (as defined in Definition \ref{def:addsets}), then $z_{S_{1}+S_{2}}^{I}\leq z_{S_{1}}^{I}+z_{S_{2}}^{I}$.
\hfill\end{lemma}

\begin{proof}
It is easy to see from the definition of $g_{S}$, when \ensuremath{S_{1},S_{2}} are collision-free, we have:

\[
g_{S_{1}+S_{2}}(t)=g_{S_{1}}(t)\cdot g_{S_{2}}(t).
\]

Therefore we obtain:

\begin{align*}
z_{S_{1}+S_{2}}^{I} & =\frac{1}{2\pi}\intop_{I}\sqrt{\mathcal{I}(g_{S_{1}+S_{2}}(t))}\mathrm{d}t=\frac{1}{2\pi}\intop_{I}\sqrt{\mathcal{I}(g_{S_{1}}(t)\cdot g_{S_{2}}(t))}\mathrm{d}t\\
 & \leq\frac{1}{2\pi}\intop_{I}\sqrt{\mathcal{I}(g_{S_{1}}(t))}\mathrm{d}t+\frac{1}{2\pi}\intop_{I}\sqrt{\mathcal{I}(g_{S_{2}}(t))}\mathrm{d}t\tag{Follows from \ref{lem:Ig1g2} }\\
 & =z_{S_{1}}^{I}+z_{S_{2}}^{I}.
\end{align*}
\hfill\end{proof}

\begin{corollary}
If $S=\{0,1,\dots,k\}\cup\{n-k,n-k+1,\dots,n\}$ with $n>2k$, then
$z_{S_{1}+S_{2}}\leq O(\log k)$.
\end{corollary}

\begin{proof}
Note that $S=\{0,1,\dots,k\}+\{0,n-k\}$. Now the result follows by
using Theorems \ref{theorem:densekostlanbound} and  \ref{theorem:sparseburgisserbound},
and Lemma \ref{lem:sumofsets}. 
\hfill\end{proof}

\subsection{Union of sets}\label{sec:setunionroots}

In subsection \ref{subsec:setadditionsroots}, we demonstrated an upper bound
$z_{S_{1}+S_{2}}$ in terms of $z_{S_{1}}$ and $z_{S_{2}}$. In this
section, we want to find upper bounds for $z_{S_{1}\sqcup S_{2}}$,
here $S_{1}\sqcup S_{2}$ denotes the disjoint union of $S_{1}$ and
$S_{2}$. First we state the following proposition which is easy to verify.
\begin{proposition}
\label{lem:disjoints1s2}If $S_{1},S_{2}\subseteq\N$ are two disjoint
sets then $g_{S_{1}\sqcup S_{2}}(t)=g_{S_{1}}(t)+g_{S_{2}}(t)$.
\end{proposition}

We need the following definition to give our result for expressing $z_{S_1 \sqcup S_2}$ in terms of $z_{S_1}$ and $z_{S_2}$.

\begin{definition}
\label{def:commongs1gs2defns}If $S_{1},S_{2}\subseteq\N$ are two
disjoint sets with $\pri{\left(\frac{g_{S_{1}}}{g_{S_{2}}}\right)}$
being non-negative at zero.\footnote{Note that at least one of $\pri{\left(\frac{g_{S_{1}}}{g_{S_{2}}}\right)}$ and $\pri{\left(\frac{g_{S_{2}}}{g_{S_{1}}}\right)}$ has to be non-negative at zero. Thus, we can always rename accordingly $S_{1}$ and $S_{2}$ to ensure this is the case.} Let $c_1, \ldots, c_m$ (with $c_{i}\leq c_{i+1}$)
be the critical points of odd multiplicity of $\frac{g_{S_{1}}}{g_{S_{2}}}$ in $(0,1)$.
Define $c_{0}:=0\text{ and }c_{m+1}:=1$. We define the following
quantities, here $0\leq i\leq m$ and $c\in(0,1)$.

\begin{align*}
\gamma_{S_{1},S_{2}}(c) & =\sqrt{\frac{g_{s_{1}}(c)}{g_{s_{2}}(c)}}\\
T_{S_{1},S_{2}}^{i} & :=\begin{cases}
\arctan(\gamma_{S_{1},S_{2}}(c_{i+1}))-\arctan(\gamma_{S_{1},S_{2}}(c_{i})) & \text{If }i\text{ is even}\\
\arctan\left(\frac{1}{\gamma_{S_{1},S_{2}}(c_{i+1})}\right)-\arctan\left(\frac{1}{\gamma_{S_{1},S_{2}}(c_{i})}\right) & \text{If }i\text{ is odd}
\end{cases}\\
R_{S_{1},S_{2}} & :=\sum_{i=0}^{m}T_{S_{1},S_{2}}^{i}.
\end{align*}
\end{definition}

\begin{lemma}
\label{lem:unions1s2correctbound}Let $S_{1},S_{2}\subseteq\N$ be
two disjoint sets. WLOG assume that $\pri{\left(\frac{g_{S_{1}}}{g_{S_{2}}}\right)}$
is non-negative at zero. Then we have: 
\[
z_{S_{1}\sqcup S_{2}}\leq z_{S_{1}}+z_{S_{2}}+\frac{1}{\pi}R_{\ensuremath{S_{1},S_{2}}}.
\]
\end{lemma}

\begin{proof}
By using Proposition \ref{lem:disjoints1s2}, we know that:

\begin{align*}
\mathcal{I}(g_{S_{1}\sqcup S_{2}}) & =\mathcal{I}(g_{S_{1}}+g_{S_{2}})=\frac{\dpri{g_{S_{1}}}+\dpri{g_{S_{2}}}}{g_{S_{1}}+g_{S_{2}}}-\left(\frac{\pri{g_{S_{1}}}+\pri{g_{S_{2}}}}{g_{S_{1}}+g_{S_{2}}}\right)^{2}+\frac{1}{t}\left(\frac{\pri{g_{S_{1}}}+\pri{g_{S_{2}}}}{g_{S_{1}}+g_{S_{2}}}\right)\\
 & =\frac{g_{S_{1}}}{g_{S_{1}}+g_{S_{2}}}\cdot \mathcal{I}(g_{S_{1}})+\frac{g_{S_{2}}}{g_{S_{1}}+g_{S_{2}}}\cdot \mathcal{I}(g_{S_{2}})+\frac{1}{g_{S_{1}}g_{S_{2}}}\left(\frac{g_{S_{1}}\pri{g_{S_{2}}}-g_{S_{2}}\pri{g_{S_{1}}}}{g_{S_{1}}+g_{S_{2}}}\right)^{2}\tag{Follows by applying Proposition \ref{lem:abcdidentity} on \ensuremath{\pri{g_{S_1}}=a,g_{S_1}=b,\pri{g_{S_2}}=c,g_{S_2}=d}}.
\end{align*}

Therefore we have:
\begin{align*}
z_{S_{1}\sqcup S_{2}} & =\frac{1}{2\pi}\intop_{0}^{1}\sqrt{\mathcal{I}(g_{S_{1}\sqcup S_{2}}(t))}\mathrm{d}t\\
 & =\frac{1}{2\pi}\intop_{0}^{1}\sqrt{\frac{g_{S_{1}}}{g_{S_{1}}+g_{S_{2}}}\cdot \mathcal{I}(g_{S_{1}})+\frac{g_{S_{2}}}{g_{S_{1}}+g_{S_{2}}}\cdot \mathcal{I}(g_{S_{2}})+\frac{1}{g_{S_{1}}g_{S_{2}}}\cdot\frac{(g_{S_{2}}\pri{g_{S_{1}}}-g_{S_{1}}\pri{g_{S_{2}}})^2}{(g_{S_{1}}+g_{S_{2}})^2}}\mathrm{d}t\\
 & \leq\frac{1}{2\pi}\cdot\left(\intop_{0}^{1}\sqrt{\mathcal{I}(g_{S_{1}}(t))}\mathrm{d}t+\intop_{0}^{1}\sqrt{\mathcal{I}(g_{S_{2}}(t))}\mathrm{d}t+\intop_{0}^{1}\big|\frac{1}{\sqrt{g_{S_{1}}g_{S_{2}}}}\cdot\frac{g_{S_{2}}\pri{g_{S_{1}}}-g_{S_{1}}\pri{g_{S_{2}}}}{g_{S_{1}}+g_{S_{2}}}\big|\mathrm{d}t\right)\\
 & =z_{S_{1}}+z_{S_{2}}+\frac{1}{2\pi}\intop_{0}^{1}\big|\frac{1}{\sqrt{g_{S_{1}}g_{S_{2}}}}\left(\frac{g_{S_{2}}\pri{g_{S_{1}}}-g_{S_{1}}\pri{g_{S_{2}}}}{g_{S_{1}}+g_{S_{2}}}\right)\big|\mathrm{d}t.
\end{align*}
Now we just need to upper bound the definite integral: $$J:=\intop_{0}^{1}\Big|{\frac{1}{\sqrt{g_{S_{1}}g_{S_{2}}}}\left(\frac{g_{S_{2}}\pri{g_{S_{1}}}-g_{S_{1}}\pri{g_{S_{2}}}}{g_{S_{1}}+g_{S_{2}}}\right)}\Big|\mathrm{d}t.$$
The value of $J$ in a sub-interval $(\alpha,\beta)$ of $(0,1)$ depends
upon the condition whether $g_{S_{2}}\pri{g_{S_{1}}}-g_{S_{1}}\pri{g_{S_{2}}}$
is positive or negative in $(\alpha,\beta)$. So we divide $(0,1)$
in the intervals where $g_{S_{2}}\pri{g_{S_{1}}}-g_{S_{1}}\pri{g_{S_{2}}}$
is positive or negative. Note that $g_{S_{2}}\pri{g_{S_{1}}}-g_{S_{1}}\pri{g_{S_{2}}}$
is positive if and only if $\pri{\left(\frac{g_{S_{1}}}{g_{S_{2}}}\right)}$
is positive. Therefore $g_{S_{2}}\pri{g_{S_{1}}}-g_{S_{1}}\pri{g_{S_{2}}}$
changes sign exactly on the critical points of odd multiplicity of $\frac{g_{S_{1}}}{g_{S_{2}}}$.
Suppose $(\alpha,\beta)$ is some sub-interval of $(0,1)$ where $\pri{\left(\frac{g_{S_{1}}}{g_{S_{2}}}\right)}$
is non-negative. Let us look at the integral $J$ in the interval
($\alpha,\beta).$ We have:

\[
J_{\alpha,\beta}:=\intop_{\alpha}^{\beta}\frac{1}{\sqrt{g_{S_{1}}g_{S_{2}}}}\left(\frac{g_{S_{2}}\pri{g_{S_{1}}}-g_{S_{1}}\pri{g_{S_{2}}}}{g_{S_{2}}^{2}}\right)\cdot\left(\frac{g_{S_{2}}^{2}}{g_{S_{1}}+g_{S_{2}}}\right)\mathrm{d}t
\]

Now we use the substitution $u=\sqrt{\frac{g_{s_{1}}}{g_{s_{2}}}}$
to obtain:

\begin{align*}
J_{\alpha,\beta} & :=\intop_{\alpha}^{\beta}\frac{1}{\sqrt{g_{S_{1}}g_{S_{2}}}}\left(\frac{g_{S_{2}}\pri{g_{S_{1}}}-g_{S_{1}}\pri{g_{S_{2}}}}{g_{S_{2}}^{2}}\right)\cdot\left(\frac{g_{S_{2}}^{2}}{g_{S_{1}}+g_{S_{2}}}\right)\mathrm{d}t\\
 & =\intop_{\alpha}^{\beta}\sqrt{\frac{g_{S_{2}}}{g_{S_{1}}}\cdot}\left(\frac{g_{S_{2}}\pri{g_{S_{1}}}-g_{S_{1}}\pri{g_{S_{2}}}}{g_{S_{2}}^{2}}\right)\cdot\left(\frac{g_{S_{2}}}{g_{S_{1}}+g_{S_{2}}}\right)\mathrm{d}t \\ &=2\intop_{\alpha}^{\beta}\pri{\left(\sqrt{\frac{g_{s_{1}}}{g_{s_{2}}}}\right)}\cdot\left(\frac{1}{1+\left(\sqrt{\frac{g_{s_{1}}}{g_{s_{2}}}}\right)^{2}}\right)\mathrm{d}t
  =2\intop_{\gamma}^{\eta}\left(\frac{1}{1+u^{2}}\right)\mathrm{d}u\tag{Here \ensuremath{\gamma=\sqrt{\frac{g_{s_{1}}(\alpha)}{g_{s_{2}}(\alpha)}}} and \ensuremath{\eta}=\ensuremath{\sqrt{\frac{g_{s_{1}}(\beta)}{g_{s_{2}}(\beta)}}}.}
\end{align*}

Therefore $J_{\alpha,\beta}=2(\arctan(\eta)-\arctan(\gamma))$ with
$\gamma=\sqrt{\frac{g_{s_{1}}(\alpha)}{g_{s_{2}}(\alpha)}}$ and $\eta=\sqrt{\frac{g_{s_{1}}(\beta)}{g_{s_{2}}(\beta)}}$.
For intervals where $\pri{\left(\frac{g_{S_{1}}}{g_{S_{2}}}\right)}$ is negative, we obtain the same result by using the substitution $u=\sqrt{\frac{g_{s_{2}}}{g_{s_{1}}}}$ instead, which is reflected on the definition of $T_{S_{1},S_{2}}^{i}$ above. Now the claimed inequality for $z_{S_{1}\sqcup S_{2}}$ follows by using the quantities defined in Definition \ref{def:commongs1gs2defns}.
\hfill\end{proof}
\begin{corollary}
\label{cor:mcriticalbound}Let $S_{1},S_{2}\subseteq\N$ be two disjoint
sets. WLOG assume that $\pri{\left(\frac{g_{S_{1}}}{g_{S_{2}}}\right)}$
is non-negative at zero. If $\frac{g_{S_{1}}}{g_{S_{2}}}$ has $m$
critical points in $(0,1)$ of odd multiplicity, then:
\[
z_{S_{1}\sqcup S_{2}}\leq z_{S_{1}}+z_{S_{2}}+\frac{m+1}{2}.
\]
\end{corollary}

\begin{proof}
By using Lemma \ref{lem:unions1s2correctbound}, we know that:
\[
z_{S_{1}\sqcup S_{2}}\leq z_{S_{1}}+z_{S_{2}}+\frac{1}{\pi}R_{S_{1},S_{2}}.
\]

Also, $R_{S_{1},S_{2}}:=\sum_{i=0}^{m}T_{S_{1},S_{2}}^{i}.$ By
using the definition of $T_{S_{1},S_{2}}^{i}$ defined in Definition \ref{def:commongs1gs2defns},
it is clear that $T_{S_{1},S_{2}}^{i}\leq\frac{\pi}{2}$. Therefore
we have: $R_{S_{1},S_{2}}\leq(m+1)\cdot\frac{\pi}{2}$. Hence the
claimed inequality follows.
\hfill\end{proof}

\section{Proof of Theorem \ref{theorem:main}: $O(\sqrt{k})\ $bound} \label{sqrtkbound}

\begin{proposition}
For any singleton set $S$, we have $\mathcal{I}(g_{S})=0$.
\end{proposition}
\begin{proof}
Suppose $S=\{a\}$, therefore $g_{S}(t)=t^{2a}$. Hence:
\begin{align*}
\mathcal{I}(g_{S}) & =\dpri{\left(\log(g_{S}(t))\right)}+\frac{\pri{\left(\log(g_{S}(t))\right)}}{t}=\dpri{(2a\log(t))}+\frac{\pri{\left(2a\log(t)\right)}}{t}\\
 & =\frac{-2a}{t^{2}}+\frac{2a}{t^{2}}=0.
\end{align*}
\hfill\end{proof}

\begin{lemma}
\label{lem:randomrealrootcardinalitytwo}For all sets $S$ of size
two, $z_{S}=\frac{1}{4}$.
\end{lemma}

\begin{proof}
WLOG we can assume that $S=\{0,a\}$. We have: 

\[
z_{S}=\frac{1}{2\pi}\intop_{I}\sqrt{\mathcal{I}(g_{S}(t))}\mathrm{d}t.
\]

An easy calculation shows that $\sqrt{\mathcal{I}(g_{S}(t))}=\frac{2at^{a-1}}{1+t^{2a}}.$
Therefore: 
\[
z_{S}=\frac{2}{2\pi}\intop_{0}^{1}\frac{at^{a-1}}{1+t^{2a}}\mathrm{d}t=\frac{1}{4}.
\]
\hfill\end{proof}

Now we show that if we increase the sparsity of a polynomial by adding a monomial of degree higher than the degree of the polynomial, we can bound the expected number of real zeros of the resulting polynomial in terms of the bound on the expected number of zeros of the original polynomial.
\begin{lemma}
\label{lem:increasekbyone}Let $S\subseteq\N$ be a set with $0\in S$
and $\card S=k$. If $a\in\N$ is such that $a>\max(S)$ then:
\[
z_{S\cup\{a\}}\leq z_{S}+\frac{1}{\pi}\arctan\left(\frac{1}{\sqrt{k}}\right).
\]
\end{lemma}

\begin{proof}
Let us first analyze the derivative of $\frac{g_{\{a\}}}{g_{S}}$.We have:

\begin{align}\label{gabygs_icreasing}
\pri{\left(\frac{g_{\{a\}}}{g_{S}}\right)} & =\frac{\left(g_{\{a\}}\right)'g_{S}-\left(g_{S}\right)'g_{\{a\}}}{g_{S}^{2}}=\frac{1}{g_{S}^{2}}\left(2ax^{2a-1}\sum_{e\in S}x^{2e}-x^{2a}\sum_{e\in S}2ex^{2e-1}\right)\\
&=\frac{2x^{2a-1}}{g_{S}^{2}}\left(\sum_{e\in S}(a-e)x^{2e}\right)>0.
\end{align}

Therefore $\frac{g_{\{a\}}}{g_{S}}$ is always increasing in $(0,1)$.
Hence we have:
\begin{align*}
z_{S\cup\{a\}} & =\frac{1}{2\pi}\intop_{0}^{1}\sqrt{\mathcal{I}(g_{S\cup\{a\}}(t))}\mathrm{d}t\\
 & =\frac{1}{2\pi}\intop_{0}^{1}\sqrt{\frac{g_{S}}{g_{S}+g_{\{a\}}}\cdot \mathcal{I}(g_{S})+\frac{g_{\{a\}}}{g_{S}+g_{\{a\}}}\cdot \mathcal{I}(g_{\{a\}})+\frac{1}{g_{S}g_{\{a\}}}\left(\frac{\pri{\left(g_{\{a\}}\right)}g_{S}-\pri{\left(g_{S}\right)}g_{\{a\}}}{g_{S}+g_{\{a\}}}\right)^{2}}\mathrm{d}t\\
 & \leq\frac{1}{2\pi}\cdot\left(\intop_{0}^{1}\sqrt{\mathcal{I}(g_{S}(t))}\mathrm{d}t+0+\intop_{0}^{1}\frac{1}{\sqrt{g_{S}g_{\{a\}}}}\left(\frac{\pri{\left(g_{\{a\}}\right)}g_{S}-\pri{\left(g_{S}\right)}g_{\{a\}}}{g_{S}+g_{\{a\}}}\right)\mathrm{d}t\right)\\
 & =z_{S}+\frac{1}{2\pi}\intop_{0}^{1}\frac{1}{\sqrt{g_{S}g_{\{a\}}}}\left(\frac{\pri{\left(g_{\{a\}}\right)}g_{S}-\pri{\left(g_{S}\right)}g_{\{a\}}}{g_{S}+g_{\{a\}}}\right)\mathrm{d}t.
\end{align*}

Now we use the substitution $u=\sqrt{\frac{g_{\{a\}}}{g_{s}}}$ to
obtain:

\begin{align*}
\intop_{0}^{1}\frac{1}{\sqrt{g_{S}g_{\{a\}}}}\left(\frac{\pri{\left(g_{\{a\}}\right)}g_{S}-\pri{\left(g_{S}\right)}g_{\{a\}}}{g_{S}+g_{\{a\}}}\right)\mathrm{d}t & =2\intop_{\alpha}^{\beta}\left(\frac{1}{1+u^{2}}\right)\mathrm{d}u\tag{Here \ensuremath{\alpha=\sqrt{\frac{g_{\{a\}}(0)}{g_{s}(0)}}=0} and \ensuremath{\beta}=\ensuremath{\sqrt{\frac{g_{\{a\}}(1)}{g_{s}(1)}}}=\ensuremath{\frac{1}{\sqrt{k}}}.}\\
 & =2\left(\arctan\left(\frac{1}{\sqrt{k}}\right)-\arctan\left(0\right)\right) \\ &=2\arctan\left(\frac{1}{\sqrt{k}}\right).
\end{align*}

Hence: 
\[
z_{S\cup\{a\}}\leq z_{S}+\frac{1}{2\pi}\intop_{0}^{1}\frac{1}{\sqrt{g_{S}g_{\{a\}}}}\left(\frac{\pri{\left(g_{\{a\}}\right)}g_{S}-\pri{\left(g_{S}\right)}g_{\{a\}}}{g_{S}+g_{\{a\}}}\right)\mathrm{d}t=z_{S}+\frac{1}{\pi}\arctan\left(\frac{1}{\sqrt{k}}\right).
\]
\hfill\end{proof}
\begin{theorem}[Theorem \ref{theorem:main} restated]\label{theorem1_restated}
\label{thm:sqrtkbound} Let $S\subseteq\N$ be a set with $0\in S$
and $\card S=k$ . Then $z_{S}\leq\frac{1}{4}+\frac{2}{\pi}(\sqrt{k-1}-1)\leq\frac{2}{\pi}\cdot\left(\sqrt{k-1}\right)$. 
\end{theorem}

\begin{proof}
If $k\leq2$ then the results follows from Lemma \ref{lem:randomrealrootcardinalitytwo}.
So assume $k>2$. By using Lemma \ref{lem:randomrealrootcardinalitytwo}
and Lemma \ref{lem:increasekbyone}, we obtain that:
\begin{equation*}
    z_{S}\leq\frac{1}{4}+\frac{1}{\pi}\sum_{i=2}^{k-1}\arctan\left(\frac{1}{\sqrt{i}}\right).\tag{We always add the highest element iteratively}
\end{equation*}

We use the following well known inequality: 
\begin{align*}
\arctan(x) & <x\text{ for all }x>0.
\end{align*}

This implies that:
\[
z_{S}\leq\frac{1}{4}+\frac{1}{\pi}\sum_{i=2}^{k-1}\frac{1}{\sqrt{i}}.
\]

Now notice that:
\[
\sum_{i=2}^{k-1}\frac{1}{\sqrt{i}}\leq\int_{1}^{k-1}\sqrt{\frac{1}{x}}\mathrm{d}x=2(\sqrt{k-1}-1).
\]

Hence the claimed bound follows.
\hfill\end{proof}

\section{Roots concentrate around 1: Proof of Theorem \ref{theorem:rootsfarmfromone}}

Here we want to show that most of the roots are near 1. First we need the following proposition useful in the analysis.
\begin{proposition}
\label{claim:easyclaim}For all $t\in(0,1),$ we have $\sqrt{\sum_{e>0}e^{2}t{}^{2e-2}}\leq\frac{1}{1-t^{2}}+\frac{2t}{(1-t^{2})^{\frac{3}{2}}}$.
\end{proposition}

\begin{proof}
First use the following well known equality:

\[
\frac{1}{1-t^{2}}=\sum_{e\geq0}t^{2e}.
\]

Using this, we obtain that:

\[
\dpri{\left(\frac{1}{1-t^{2}}\right)}=\sum_{e>0}2e(2e-1)t^{2e-2}=\frac{2(1+3t^{2})}{(1-t^{2})^{3}}.
\]

Therefore:

\[
\sum_{e>0}e(2e-1)t^{2e-2}=\frac{(1+3t^{2})}{(1-t^{2})^{3}}.
\]

Clearly:

\begin{align*}
\sqrt{\sum_{e>0}e^{2}t{}^{2e-2}} & \leq\sqrt{\sum_{e>0}e(2e-1)t^{2e-2}}\leq\sqrt{\frac{(1+3t^{2})}{(1-t^{2})^{3}}}=\sqrt{\frac{1}{(1-t^{2})^{2}}+\frac{4t^{2}}{(1-t^{2})^{3}}}\\
 & \leq\frac{1}{1-t^{2}}+\frac{2t}{(1-t^{2})^{\frac{3}{2}}}.
\end{align*}
\hfill\end{proof}

We now give the proof of Theorem \ref{theorem:rootsfarmfromone}.

\begin{proof} (Proof of Theorem \ref{theorem:rootsfarmfromone}). 
WLOG, we can assume that $0\in S$, therefore $\norm{v_{S}(t)}_{2}\geq1$
for all $t\in\R$. By using the equality in Theorem \ref{theorem:realrootsboundkostlan} and also by ignoring the second term in \eqref{zsreqinte}, we get the following
inequality for $z_{S}$:
\begin{align*}
z_{S}^{(0,1-\epsilon)} & \leq\frac{1}{\pi}\intop_{0}^{1-\epsilon}\frac{\sqrt{(\norm{v_{S}(t)}_{2}\cdot\norm{\pri{v_{S}}(t)}_{2})^{2}}}{(\norm{v_{S}(t)}_{2})^{2}}\mathrm{d}t=\frac{1}{\pi}\intop_{0}^{1-\epsilon}\frac{\norm{\pri{v_{S}}(t)}_{2}}{\norm{v_{S}(t)}_{2}}\mathrm{d}t\\
 & \leq\frac{1}{\pi}\intop_{0}^{1-\epsilon}\norm{\pri{v_{S}}(t)}_{2}\mathrm{d}t
\end{align*}

By using \ref{claim:easyclaim}, we have: $\norm{\pri{v_{S}}(t)}_{2}=\sqrt{\sum_{e\in S}e^{2}t^{2e-2}}\leq\frac{1}{1-t^{2}}+\frac{2t}{(1-t^{2})^{\frac{3}{2}}}$.
Therefore: 

\begin{align*}
z_{S}^{(0,1-\epsilon)} & \leq\frac{1}{\pi}\intop_{0}^{1-\epsilon}\norm{\pri{v_{S}}(t)}_{2}\mathrm{d}t\leq\frac{1}{\pi}\intop_{0}^{1-\epsilon}\left(\frac{1}{1-t^{2}}+\frac{2t}{(1-t^{2})^{\frac{3}{2}}}\right)\mathrm{d}t\\
 & =\frac{1}{\pi}\left(\intop_{0}^{1-\epsilon}\frac{1}{1-t^{2}}\mathrm{d}t+\frac{1}{\pi}\intop_{0}^{1-\epsilon}\frac{2t}{(1-t^{2})^{\frac{3}{2}}}\mathrm{d}t\right) \\ &=\frac{1}{\pi}\left(\left[\frac{1}{2} \log\left(\frac{1+t}{1-t}\right)\right]_{0}^{1-\epsilon}+\left[\frac{2}{\sqrt{1-t^{2}}}\right]_{0}^{1-\epsilon}\right)\\
 & =\frac{1}{\pi}\left(\frac{1}{2}\log\left(\frac{2-\epsilon}{\epsilon}\right)+\frac{2}{\sqrt{\epsilon(2-\epsilon)}}-2\right)\leq\frac{1}{2\pi}\left(\log\left(\frac{2}{\epsilon}\right)+\frac{4}{\sqrt{\epsilon}}-4\right).
\end{align*}
\hfill\end{proof}

\section{The lower bound}\label{lb}
In this section we will come up with a sequence of sets $(S_k)_{k\geq1}$ such that the expected number of real zeros of the corresponding polynomials is lower bounded by $\Omega(\sqrt{k})$, for large enough  $k$.

\begin{lemma}\label{lem:intI}
Suppose $S=\{e_1,e_2,\dots,e_k\}$ with $e_k=\max(S)$ and $b\geq 1$. If $I_1=\intop_{1-\frac{1}{b}}^{1}\sqrt{\mathcal{I}(g_{S})}\mathrm{d}t$, then $I_1\leq \frac{2(k+1)e_k}{b}.$
\end{lemma}
\begin{proof}
We have:
\begin{align*}
    \sqrt{\mathcal{I}(g_{S})}&=2\frac{\sqrt{\sum_{c}\sum_{\substack{i<j\\e_{i}+e_{j}-1=c}}(e_{i}-e_{j})^2t^{2c}}}{g_S}\tag{By using Proposition \ref{prop:edelmanreformulation}}\\
    &\leq 2\frac{\sqrt{\sum_{c}\sum_{\substack{i<j\\e_{i}+e_{j}-1=c}}(e_{k+1})^2}}{g_S}\\
    &\leq 2\frac{\sqrt{(k+1)^2e_k^2}}{g_S}\\
    &\leq 2(k+1)e_k.
\end{align*}
Therefore, we have: $I_1\leq \intop_{1-\frac{1}{b}}^{1}2(k+1)e_k=\frac{2(k+1)e_k}{b}.$
\hfill\end{proof}
\begin{remark}\label{rem:arbitsmall}
In the view of the above lemma, we can have $z_S^I$ arbitrarily small, with $I=(1-\frac{1}{b},1)$ for a large enough $b$. This fact will be crucial in the proof of Theorem \ref{theorem:lowerbound}. Further, Lemma \ref{lem:intI} can be viewed as a supplementary result to Theorem \ref{theorem:rootsfarmfromone}. Theorem \ref{theorem:rootsfarmfromone} implies that most of the roots lie in $(0,1-\epsilon)$, if $\epsilon$ is allowed to be arbitrarily small. Lemma \ref{lem:intI} gives a precise formulation of this fact. 
\end{remark}

From now on we will assume that $S=\{0,1\}\bigcup\{2^{2^i}\ |\ 1\leq i\leq k-1\}$ and $a=2^{2^{k}}.$ The following lemma essentially will imply that one cannot avoid summing over $\sqrt{\frac{1}{k}}$ as in the proof of Theorem \ref{theorem1_restated}. 
\begin{lemma}\label{lem:intW}
Let $W=\frac{1}{g_{S}g_{\{a\}}}\left(\frac{\pri{\left(g_{\{a\}}\right)}g_{S}-\pri{\left(g_{S}\right)}g_{\{a\}}}{g_{S}+g_{\{a\}}}\right)^{2}$, then $\intop_{1-\frac{1}{2a}}^{1}\sqrt{W}\mathrm{d}t\geq \frac{c}{\sqrt{k}}$ for some real constant $c>0$.
\end{lemma}
\begin{proof}
Using the computation in the proof of Lemma \ref{lem:increasekbyone} we have:
$$
\intop_{1-\frac{1}{2a}}^{1}\sqrt{W}\mathrm{d}t
 =2\left(\arctan\left(\frac{1}{\sqrt{k+1}}\right)-\arctan\left(\sqrt{\frac{g_{\{a\}}(1-\frac{1}{2a})}{g_S(1-\frac{1}{2a})}}\right)\right).
$$
We now upper bound the value of $\arctan\left(\sqrt{\frac{g_{\{a\}}(1-\frac{1}{2a})}{g_S(1-\frac{1}{2a})}}\right)$ by giving a lower bound on $g_S(1-\frac{1}{2a})$ and an upper bound on $g_{\{a\}}(1-\frac{1}{2a}).$ Using the well known inequalities $\left(1-\frac{1}{n}\right)^{n}\leq\frac{1}{e}$ (for any $n\in\N$) and $$(1+x)^r\geq 1+rx\ \mathrm{if}\ x\geq -1\ \mathrm{and}\ r>1,$$ we have, for large enough $k$:
\begin{equation}\label{eq:gaps}
g_S(1-\frac{1}{2a})=\sum\limits_{i=1}^{k+1} \left(1-\frac{1}{2a}\right)^{2e_i}\geq \sum\limits_{i=1}^{k+1} \left(1-\frac{2e_i}{2a}\right)\geq k+1-\left(\sum\limits_{i=1}^{k+1}2^{-k}\right)\geq k
\end{equation}
Therefore, 
\begin{align*}
    \arctan\left(\sqrt{\frac{g_{\{a\}}(1-\frac{1}{2a})}{g_S(1-\frac{1}{2a})}}\right)&\leq\arctan\left(\sqrt{\frac{\frac{1}{e}}{k}}\right)\\
    \Rightarrow 2\left(\arctan\left(\frac{1}{\sqrt{k+1}}\right)-\arctan\left(\sqrt{\frac{g_{\{a\}}(1-\frac{1}{2a})}{g_S(1-\frac{1}{2a})}}\right)\right)&\geq  2\arctan\left(\frac{1}{\sqrt{k+1}}\right)\\
    &\hspace{.15in} -2\arctan\left(\sqrt{\frac{\frac{1}{e}}{k}}\right)\\
    &\geq2\left(\arctan\left(\frac{\frac{1}{\sqrt{k+1}}-\frac{1}{{e\sqrt{k}}}}{1+\frac{1}{e\sqrt{k(k+1)}}}\right)\right) \\
    &=2\left(\arctan(\frac{c'}{\sqrt{k}})\right) \tag{for some $c'>0$}.
\end{align*}
\hfill\end{proof}

\subsection{Proof of Theorem \ref{theorem:lowerbound}}
For proving Theorem \ref{theorem:lowerbound} we will again resort to our idea of monomial-wise construction of polynomial. The monomial sequence we choose is $e_{i+2}=2^{2^{i}}$ for $i\geq 1$ with $e_1=0,e_2=1$. Before we begin the proof, recall from the proof of Lemma \ref{lem:increasekbyone} that \begin{align*}
    z_{S\cup\{a\}}&=\frac{1}{2\pi}\intop_{0}^{1}\sqrt{\frac{g_{S}}{g_{S}+g_{\{a\}}}\cdot \mathcal{I}(g_{S})+\frac{g_{\{a\}}}{g_{S}+g_{\{a\}}}\cdot \mathcal{I}(g_{\{a\}})+\frac{1}{g_{S}g_{\{a\}}}\left(\frac{\pri{\left(g_{\{a\}}\right)}g_{S}-\pri{\left(g_{S}\right)}g_{\{a\}}}{g_{S}+g_{\{a\}}}\right)^{2}}\mathrm{d}t\\
    &=\frac{1}{2\pi}\intop_{0}^{1}\sqrt{\frac{g_{S}}{g_{S}+g_{\{a\}}}\cdot \mathcal{I}(g_{S})+\frac{g_{\{a\}}}{g_{S}+g_{\{a\}}}\cdot \mathcal{I}(g_{\{a\}})+W}\mathrm{d}t\tag{using the notation in Lemma \ref{lem:intW}}.
\end{align*}
The key idea is to write $z_{S\cup\{a\}}$ as a sum of two integrals over disjoint intervals such that $\mathcal{I}(g_S)$ dominates in one interval while $W$ dominates in the other. The rest of the proof  is about proving lower bounds on these two integrals.

\begin{proof}
We have:
\begin{align*}
z_{S\cup\{a\}} & =\frac{1}{2\pi}\intop_{0}^{1}\sqrt{\frac{g_{S}}{g_{S}+g_{\{a\}}}\cdot \mathcal{I}(g_{S})+\frac{g_{\{a\}}}{g_{S}+g_{\{a\}}}\cdot \mathcal{I}(g_{\{a\}})+W}\mathrm{d}t\\
 & =\frac{1}{2\pi}\cdot\left(\intop_{0}^{1}\sqrt{\frac{g_{S}}{g_{S}+g_{\{a\}}}\cdot \mathcal{I}(g_{S})+0+W}\mathrm{d}t\right)\\
    &=\frac{1}{2\pi}\left(\intop_{0}^{1-\frac{1}{2a}}\sqrt{\frac{g_{S}}{g_{S}+g_{\{a\}}}\cdot \mathcal{I}(g_{S})+W}\mathrm{d}t+\intop_{1-\frac{1}{2a}}^{1}\sqrt{\frac{g_{S}}{g_{S}+g_{\{a\}}}\cdot \mathcal{I}(g_{S})+W}\mathrm{d}t\right)\\
    &\geq \frac{1}{2\pi}\left(\intop_{0}^{1-\frac{1}{2a}}\sqrt{\frac{g_{S}}{g_{S}+g_{\{a\}}}\cdot \mathcal{I}(g_{S})}\mathrm{d}t+\intop_{1-\frac{1}{2a}}^{1}\sqrt{W}\mathrm{d}t\right)\\
    &=\frac{1}{2\pi}\left(\intop_{0}^{1}\sqrt{\frac{g_{S}}{g_{S}+g_{\{a\}}}\cdot \mathcal{I}(g_{S})}\mathrm{d}t-\intop_{1-\frac{1}{2a}}^{1}\sqrt{\frac{g_{S}}{g_{S}+g_{\{a\}}}\cdot \mathcal{I}(g_{S})}\mathrm{d}t+\intop_{1-\frac{1}{2a}}^{1}\sqrt{W}\mathrm{d}t\right)\\
    &\geq \frac{1}{2\pi}\left(\sqrt{\frac{k+1}{k+2}}\cdot\intop_{0}^{1}\sqrt{\mathcal{I}(g_{S})}\mathrm{d}t-\intop_{1-\frac{1}{2a}}^{1}\sqrt{\mathcal{I}(g_{S})}\mathrm{d}t+\intop_{1-\frac{1}{2a}}^{1}\sqrt{W}\mathrm{d}t\right)\tag{ $\frac{g_{\{a\}}}{g_S}$ is increasing (Equation \eqref{gabygs_icreasing})}\\
    &= \sqrt{\frac{k+1}{k+2}}\cdot z_S+\frac{1}{2\pi}\left(- \intop_{1-\frac{1}{2a}}^{1}\sqrt{\mathcal{I}(g_{S})}\mathrm{d}t+\intop_{1-\frac{1}{2a}}^{1}\sqrt{W}\mathrm{d}t\right).
\end{align*}
Now by using Lemma \ref{lem:intI} with $b=2a$ and Lemma \ref{lem:intW} we have:
\begin{align*}
    z_{S\cup\{a\}}&\geq\sqrt{\frac{k+1}{k+2}}\cdot z_S+\frac{1}{2\pi}\left(\intop_{1-\frac{1}{2a}}^{1}\sqrt{W}\mathrm{d}t-I_1\right)\\
    &\geq\sqrt{\frac{k+1}{k+2}}\cdot z_S+\frac{1}{\pi}\arctan\left(\frac{c'}{\sqrt{k}}\right)-\frac{k+1}{2\pi2^{2^{k-1}}}.
\end{align*}
By a generalization of the Shafer-Fink inequality \cite[Theorem 1]{aurizio}, we have  $z_{S\cup\{a\}}\geq \sqrt{\frac{k+1}{k+2}}\cdot z_S+\frac{c''}{\sqrt{k}}$ for some $c''\in (0,1)$. Therefore, for large enough $k$, and $i$ such that $k-i$ is large:

$$z_{S\cup\{a\}}\geq \sqrt{\frac{k+1-i}{k+2}}\cdot
z_{S_{k-i}}+c''\left(\sum\limits_{j=0}^i\frac{1}{\sqrt{k-j}}\cdot\sqrt{\frac{k+2-j}{k+2}}\right),$$
where $S_{k-i}=\{2^{2^\ell}\ |\ 1\leq \ell \leq k-i-1\}\cup\{0,1\}$. Now let $i=[k/2]$, where $[x]$ denotes the greatest integer less than or equal to $x$. Hence, we have:
\begin{align*}
z_{S\cup\{a\}}&\geq \sqrt{\frac{k+1-i}{k+2}}\cdot
z_{S_{k-i}}+c''\left(\sum\limits_{j=0}^{[k/2]}\frac{1}{\sqrt{k-j}}\cdot\sqrt{\frac{k+2-j}{k+2}}\right)\\
&\geq \frac{1}{\sqrt{2}}\cdot z_{S_{[k/2]}}+c''([k/2]+1)\frac{1}{\sqrt{2}\sqrt{k}}\\
&\geq c'''\sqrt{k} \tag{for some real number $c'''$}.
\end{align*}
This proves the theorem.
\hfill\end{proof}

\section{Conclusion}

We settle the bound on the expected number of real zeros of a random $k$-sparse polynomials when the coefficients are independent standard normal random variables. We first showed an $O(\sqrt{k})$ upper bound for an arbitrary set of size $k$, and then gave an example of set where this bound is tight. We see this as another step towards understanding the number of real zeros of sparse polynomials and related generalizations. 

In this article, we considered random variables following independent standard normal distributions. It would be interesting to study other distributions on the coefficients, although we expect analysis to become increasingly difficult as the distributions become more complex.

We also mentioned how the real $\tau$-conjecture is connected to the problem we study and its importance in algebraic complexity. Towards resolving the conjecture, consider the simple setting where $f$ and $g$ are both $k$-sparse polynomials and we wish to study the number of real zeros of $f g + 1$. This is essentially the first case which is non-trivial, unfortunately very little is known and prior techniques seem to fail so far.

Also, there is a vast number of restricted arithmetic circuit models. We invite the community, especially experts on these models, to consider the number of real zeros of univariate polynomials under such restrictions and explore their connections with complexity theoretic lower bounds. It is conceivable that one can find a restriction for which the behavior of the expected number of real zeros is easier to understand than the sparse case and which may lead to new insights towards resolving the aforementioned generalizations, such as the ones considered in the real $\tau$-conjecture.

\subsection*{Acknowledgements}
We thank our advisor Markus Bl\"aser for his constant support throughout the work. We thank Vladimir Lysikov for many insightful discussions on the topic. AP thanks S\'ebastien Tavenas for hosting him at Universit\'e Savoie Mont Blanc, Chamb\'ery and for encouraging discussions there.

\bibliographystyle{amsplain}
\bibliography{Bibliography}

\section{Appendix}
\subsection{Recovering the classics: \texorpdfstring{$O(\log n)$}{The logn} bound in the dense case}\label{lognbound}

In this section, we give a simple proof of Theorem \ref{theorem:densekostlanbound}
using the tools and notations developed in Section \ref{sec:altform}. To
this end, first we prove the following lemma.
\begin{lemma}
\label{lem:nocriticalpoints}Let $S_{1}$ and $S_{2}=\{a\}$ be such
that $a>\max(S_{1})$. Then \textup{$\frac{g_{S_{1}}}{g_{S_{2}}}$
has no critical points in $(0,1)$.}
\end{lemma}

\begin{proof}
Critical points of $\frac{g_{S_{1}}}{g_{S_{2}}}$ are exactly the
zeroes of $g_{S_{2}}\pri{g_{S_{1}}}-g_{S_{1}}\pri{g_{S_{2}}}$. We
have: 

\begin{align*}
g_{S_{2}}\pri{g_{S_{1}}}-g_{S_{1}}\pri{g_{S_{2}}} & =x^{2a}\pri{g_{S_{1}}}-2ax^{2a-1}g_{S_{1}}=x^{2a-1}(x\pri{g_{S_{1}}}-2ag_{S_{1}}).\\
 & =x^{2a-1}\sum_{e\in S_{1}}(2e-2a)x^{2e}.
\end{align*}

which is clearly always negative in $(0,1)$. Thus $\frac{g_{S_{1}}}{g_{S_{2}}}$
has no critical points in (0,1).

\hfill\end{proof}

\begin{theorem}
\label{theorem:kostlanorderthm}If $S=\{0,1,2,\dots,n\}$ then $z_{S}\leq\frac{3}{4}\log_2(n)$.
\end{theorem}

\begin{proof}
We prove it by induction on $n$. The base case of $n=1$ is trivially
true. 

Suppose $n$ is odd i.e.  $n=2a+1$ for some $a \in \Z_{+}$. In this
case $S$ is the independent sum of $\{0,1,\dots,a\}$ and $\{0,a+1\}$.
Therefore by using Lemma \ref{lem:sumofsets}, we know that $z_{S}\leq z_{\{0,1,\dots,a\}}+z_{\{0,a+1\}}$.
We know, $z_{\{0,a+1\}}=\frac{1}{4}$. By using the induction hypothesis,
we know that $z_{\{0,1,\dots,a\}}\leq\frac{3}{4}\log_2(a)$. Hence $z_{S}\leq\frac{3}{4}\log_2(a)+\frac{1}{4}\leq\frac{3}{4}\log_2(2a+1)$.

Now consider the case when n is even i.e. $n=2a$ for some $a\in\Z_{+}$.
We have $S=\{0,1,\dots,2a-1\}\cup\{2a\}$. By using Lemma \ref{lem:nocriticalpoints}
and Corollary \ref{cor:mcriticalbound}, we get that

\[
z_{S}\leq z_{\{0,1,\dots,2a-1\}}+\frac{1}{2}\leq z_{\{0,1,\dots,a-1\}}+\frac{1}{4}+\frac{1}{2}\leq\frac{3}{4}\log_2(a-1)+\frac{3}{4}\leq\frac{3}{4}\log_2(2a).
\]
\hfill\end{proof}

Theorem \ref{theorem:kostlanorderthm} shows that $z_{\ensuremath{\{0,1,\dots n\}}}\leq\frac{3}{4}\log_2(n)$,
which is worse bound than Theorem \ref{theorem:densekostlanbound}. But asymptotically
they are similar.

\subsection{Proof of the $O(\sqrt{k}\log(k))$ bound \cite{burgisserergurcueto2018} [Theorem \ref{theorem:sparseburgisserbound}]}
Before giving the proof we first draw attention to the following folklore lemma about
$\ell_{1}$ and $\ell_{2}$ norms which is used in the proof.
\begin{lemma}
\label{lem:l1andl2normineq}For all $x\in\R^{k}$, we have the following
inequality between $\ell_{1}$ and $\ell_{2}$ norms of $x$:

\[
\norm x_{2}\leq\norm x_{1}\leq\sqrt{k}\norm x_{2}.
\]
\end{lemma}

\begin{proof}
The inequality $\norm x_{2}\leq\norm x_{1}$ is trivial. For the second
inequality, we use Cauchy-Schwartz to get:

\[
\norm x_{1}=\sum\limits _{i=1}^{k}|x_{i}|=\sum\limits _{i=1}^{k}|x_{i}|\cdot1\leq\left(\sum\limits _{i=1}^{k}x_{i}^2\right)^{1/2}\left(\sum\limits _{i=1}^{k}1^{2}\right)^{1/2}=\sqrt{k}\norm x_{2}.
\]
\hfill\end{proof}

The following upper bound on $z_{S}$ was proven in \cite{burgisserergurcueto2018} using the Edelman-Kostlan integral.
\begin{theorem}[Theorem 1.3 in \cite{burgisserergurcueto2018}]
Let $S\subseteq\N$ be any set as
above with $\card S=k$ then we have
\[
z_{S}\leq\frac{1}{\pi}\sqrt{k}\cdot\log(k).
\]
\end{theorem}

\begin{proof}
We use the inequality as in the proof of Theorem \ref{theorem:rootsfarmfromone}. 
\begin{align*}
z_{S} & \leq\frac{1}{\pi}\intop_{0}^{1}\frac{\sqrt{(\norm{v_{S}(t)}_{2}\cdot\norm{\pri{v_{S}}(t)}_{2})^{2}}}{(\norm{v_{S}(t)}_{2})^{2}}\mathrm{d}t=\frac{1}{\pi}\intop_{0}^{1}\frac{\norm{\pri{v_{S}}(t)}_{2}}{\norm{v_{S}(t)}_{2}}\mathrm{d}t\\
 & \leq\frac{1}{\pi}\intop_{0}^{1}\sqrt{k}\cdot\frac{\norm{\pri{v_{S}}(t)}_{1}}{\norm{v_{S}(t)}_{1}}\mathrm{d}t\tag{By using Lemma \ref{lem:l1andl2normineq}}\\
 & =\frac{1}{\pi}\sqrt{k}\cdot\left[\log(\norm{v_{S}(t)}_{1})\right]_{0}^{1}=\frac{1}{\pi}\sqrt{k}\cdot(\log(\norm{v_{S}(1)}_{1})-\log(\norm{v_{S}(0)}_{1}))\\
 & =\frac{1}{\pi}\sqrt{k}\log(k).
\end{align*}
\hfill\end{proof}


\end{document}